\documentclass[12pt]{amsart}
\usepackage{amscd,amsmath,amsthm,amssymb}
\usepackage{pstcol,pst-plot,pst-3d}
\usepackage{lmodern,pst-node}
\usepackage{multicol}
\usepackage{epic,eepic}
\usepackage[]{graphicx}
\psset{unit=0.7cm,linewidth=0.8pt,arrowsize=2.5pt 4}

\newpsstyle{fatline}{linewidth=1.5pt}
\newpsstyle{fyp}{fillstyle=solid,fillcolor=verylight}
\definecolor{verylight}{gray}{0.97}
\definecolor{light}{gray}{0.9}
\definecolor{medium}{gray}{0.85}

\numberwithin{equation}{section}

\def\NZQ{\mathbb}               

\def\ZZ{{\NZQ Z}}

\def\frk{\mathfrak}               

\def\Phi{{\frk N}}
\def\ab{{\bold a}}
\def\bb{{\bold b}}

\def\xb{{\bold x}}

\def\opn#1#2{\def#1{\operatorname{#2}}} 
\opn\chara{char} \opn\length{\ell} \opn\pd{pd} \opn\rk{rk}
\opn\projdim{proj\,dim} \opn\injdim{inj\,dim} \opn\rank{rank}
\opn\depth{depth} \opn\grade{grade} \opn\height{height}
\opn\embdim{emb\,dim} \opn\codim{codim}

\opn\Tr{Tr} \opn\bigrank{big\,rank}
\opn\superheight{superheight}\opn\lcm{lcm}
\opn\trdeg{tr\,deg}
\opn\reg{reg} \opn\lreg{lreg} \opn\ini{in} \opn\lpd{lpd}
\opn\size{size}\opn{\mult}{mult}
\opn\div{div} \opn\Div{Div} \opn\cl{cl} \opn\Cl{Cl}
\opn\Spec{Spec} \opn\Supp{Supp} \opn\supp{supp} \opn\Sing{Sing}
\opn\Ass{Ass} \opn\Min{Min}
\opn\Ann{Ann} \opn\Rad{Rad} \opn\Soc{Soc}
\opn\Syz{Syz} \opn\Im{Im} \opn\Ker{Ker} \opn\Coker{Coker}
\opn\Am{Am} \opn\Hom{Hom} \opn\Tor{Tor} \opn\Ext{Ext}
\opn\End{End} \opn\Aut{Aut} \opn\id{id} \opn\ini{in}

\opn\nat{nat}
\opn\pff{pf}
\opn\Pf{Pf} \opn\GL{GL} \opn\SL{SL} \opn\mod{mod} \opn\ord{ord}
\opn\Gin{Gin}
\opn\Hilb{Hilb}\opn\adeg{adeg}\opn\std{std}\opn\ip{infpt}
\opn\Pol{Pol}
\opn\sat{sat}
\opn\Var{Var}
\opn\Gen{Gen}

\opn\aff{aff} \opn\con{conv} \opn\relint{relint} \opn\st{st}
\opn\lk{lk} \opn\cn{cn} \opn\core{core} \opn\vol{vol}
\opn\link{link} \opn\star{star}
\opn\gr{gr}

\def\Gc{{\mathcal G}}

\def\Lc{{\mathcal L}}
\def\Mc{{\mathcal M}}
\def\Dc{{\mathcal D}}
\def\pot#1#2{#1[\kern-0.28ex[#2]\kern-0.28ex]}

\opn\dirlim{\underrightarrow{\lim}}
\opn\inivlim{\underleftarrow{\lim}}
\def\Implies{\ifmmode\Longrightarrow \else
        \unskip${}\Longrightarrow{}$\ignorespaces\fi}
\def\implies{\ifmmode\Rightarrow \else
        \unskip${}\Rightarrow{}$\ignorespaces\fi}
\def\iff{\ifmmode\Longleftrightarrow \else
        \unskip${}\Longleftrightarrow{}$\ignorespaces\fi}

\let\:=\colon
 \newtheorem{Theorem}{Theorem}[section]
 \newtheorem{Lemma}[Theorem]{Lemma}
 \newtheorem{Corollary}[Theorem]{Corollary}
 \newtheorem{Proposition}[Theorem]{Proposition}
 \newtheorem{Remark}[Theorem]{Remark}
 
 \newtheorem{Example}[Theorem]{Example}
 
 \newtheorem{Definition}[Theorem]{Definition}

\let\epsilon\varepsilon
\let\phi=\varphi
\let\kappa=\varkappa
\def\qed{\ifhmode\textqed\fi
      \ifmmode\ifinner\quad\qedsymbol\else\dispqed\fi\fi}
\def\textqed{\unskip\nobreak\penalty50
       \hskip2em\hbox{}\nobreak\hfil\qedsymbol
       \parfillskip=0pt \finalhyphendemerits=0}
\def\dispqed{\rlap{\qquad\qedsymbol}}

\opn\dis{dis}
\def\pnt{{\raise0.5mm\hbox{\large\bf.}}}

\opn\Lex{Lex}

\newcommand{\inD}[1][\relax]{\def\argone{#1}\def\temprelax{\relax}
  \ifx\argone\temprelax\right.\else\,\middle|#1\right.{}\fi}

\newif\ifbinary
\binarytrue

\begin{document}

\title{The join-meet ideal of a finite lattice}

\author{Viviana Ene and Takayuki Hibi}

\address{Viviana Ene, Faculty of Mathematics and Computer Science, Ovidius University, Bd.\ Mamaia 124,
 900527 Constanta, Romania} \email{vivian@univ-ovidius.ro}

\address{Takayuki Hibi, Department of Pure and Applied Mathematics, Graduate School of Information Science and Technology,
Osaka University, Toyonaka, Osaka 560-0043, Japan}
\email{hibi@math.sci.osaka-u.ac.jp}


\thanks{The first author was supported  by the JSPS Invitation Fellowship Programs for Research in Japan.}

 \begin{abstract}
Radical binomial ideals associated with finite lattices are studied. Gr\"obner basis theory turns out to be an efficient tool in this investigation.
 \end{abstract}

\maketitle

\section*{Introduction}

Let $L$ be a finite lattice and $K[L]$ the polynomial ring over a field $K$ whose variables are the elements of $L.$ Let $I_L$ be the {\em join-meet 
ideal} of $L,$ that is, the ideal of $K[L]$ which is generated by all the binomials of the form $f=ab-(a\wedge b)(a\vee b)$, where $a,b\in L$ are 
incomparable elements. Of course one may ask whether algebraic properties of $I_L$ are related to the combinatorial properties of $L.$  $I_L$ is a 
prime ideal if and only if $L$ is distributive as it was shown in \cite{H} and if $L$ is distributive, the Gr\"obner bases of $I_L$ with respect to 
various monomial orders have been studied; see, for instance, \cite{H}, \cite{HH1}, \cite{AHH}, \cite{Q}. In the same hypothesis on $L,$ the toric 
ring $K[L]/I_L$  is well understood; see \cite{H}, \cite{H1}, \cite{H2}, \cite{H3}.

Almost nothing is known about the join-meet ideal $I_L$ when $L$ is not distributive. In the present paper we focus on the join-meet ideals of modular 
and non-distributive lattices. For basic properties of lattices, like distributivity and modularity, we refer the reader to the well known monographs
\cite{B} and \cite{St}.

It was conjectured in \cite{HH1} that, given a modular lattice $L$, for any monomial order $<$ on $K[L]$ the initial ideal $\ini_<(I_L)$ is not 
squarefree, unless $L$ is distributive. We give a proof of this conjecture in Section~\ref{conjecture}. This result shows, in particular, that  
for deciding whether a join-meet ideal $I_L$ of a modular and non-distributive lattice $L$ is radical one cannot use the known 
statement that a polynomial ideal is radical if it has a  squarefree initial ideal. Moreover, easy examples show that even if the lattice $L$
is rather closed to a distributive lattice, the ideal $I_L$ might not be radical; see Example~\ref{latticeN}. A general 
characterization of radical join-meet ideals associated with modular non-distributive lattices seems to be difficult.
However, in Section~\ref{modular}, we find a class of modular non-distributive lattices $L$ whose join-meet ideal $I_L$ is radical. To prove this 
property we intensively use the Gr\"obner basis theory.

For radical join-meet ideals, in Section~\ref{radsection}, we describe the minimal prime ideals. This description is used later, in Section~\ref{minsection}, to obtain a complete characterization of the minimal primes of the radical join-meet ideals studied in Section~\ref{modular}.

\section{The squarefree conjecture}
\label{conjecture}

Let $L$ be a finite lattice and $K[L]$ the polynomial ring  over a field $K$ whose variables are the elements of $L.$ A binomial of $K[L]$ of the form
$f=ab-(a\wedge b)(a\vee b)$, where $a,b\in L$ are incomparable, is called a {\em basic binomial}. In some recent papers, the basic binomials are called Hibi relations.

\begin{Definition}
The {\em join-meet ideal} of $L$ is the ideal of $K[L]$ generated by the basic binomials, that is, 
\[
I_L=(ab-(a\wedge b)(a\vee b): a,b\in L, a,b \text{ incomparable })\subset K[L]. 
\]
\end{Definition}

The join-meet ideal of a lattice was introduced in \cite{H}.  For fundamental notions on lattices we refer to \cite{B} and \cite{St}. 

The main result of this section answers positively a conjecture made in \cite{HH1}. We first need a preparatory result on modular and non-distributive lattices which might be known, but we include its proof since we could not find any reference. 

\begin{Lemma}\label{small diamond}
Let $L$ be a modular non-distributive lattice. Then $L$ has a diamond sublattice $L^\prime$  such that 
$\rank \max L^\prime - \rank \min L^\prime =2.$
\end{Lemma}

\begin{proof}
Let $\delta$ be a diamond of $L$ labeled  as in Figure~\ref{small1} (i) of minimal rank, that is, $\rank e-\rank a=$ minimal. 

\begin{figure}[bht]
\begin{center}
\psset{unit=0.5cm}
\begin{pspicture}(-10.3,-2.5)(4,3.5)

\psline(-9,3)(-11,1)
\psline(-9,3)(-9,1)
\psline(-9,3)(-7,1)
\psline(-11,1)(-9,-1)
\psline(-9,1)(-9,-1)
\psline(-7,1)(-9,-1)

\rput(-9,3){$\bullet$}
\put(-9.1,3.3){$e$}
\rput(-11,1){$\bullet$}
\put(-11.6,0.9){$b$}
\rput(-9,1){$\bullet$}
\put(-8.6,0.9){$c$}
\rput(-7,1){$\bullet$}
\put(-6.7,1){$d$}
\rput(-9,-1){$\bullet$}
\put(-9.1,-1.6){$a$}

\rput(-13,-1){(i)}

\psline(1,3)(-1,1)
\psline(1,3)(1,1)
\psline(1,3)(3,2)
\psline(3,2)(3,0)
\psline(-1,1)(1,-1)
\psline(1,1)(1,-1)
\psline(3,0)(1,-1)

\rput(1,3){$\bullet$}
\put(1.1,3.3){$e$}
\rput(-1,1){$\bullet$}
\put(-1.6,0.9){$b$}
\rput(1,1){$\bullet$}
\put(1.4,0.9){$c$}
\rput(3,2){$\bullet$}
\put(3.3,2){$d$}
\rput(1,-1){$\bullet$}
\put(0.9,-1.6){$a$}
\rput(3,0){$\bullet$}
\put(3.3,0){$f$}

\rput(5,-1){(ii)}
\end{pspicture}
\end{center}
\caption{}\label{small1}
\end{figure}

We show that $\rank a-\rank e=2.$ Let us assume that $\rank e > \rank a +2.$ By duality, we may assume, for instance, that $\rank d > \rank a +1,$ 
that is, there exists $f\in L$ such that $a < f < d.$ Then we get the lattice displayed in Figure~\ref{small1} (ii) where $c\wedge f=c\wedge d=a$ and 
$c\vee f \leq c\vee d=e.$ If $c\vee f=e,$ then $L$ has a pentagon subblattice (with the elements $a,c,f,d,e$), which is impossible since $L$ is modular.
Therefore, we must have $c\vee f < e. $

\begin{figure}[bht]
\begin{center}
\psset{unit=0.5cm}
\begin{pspicture}(-5.3,-2.5)(4,3.5)

\psline(6,3)(4,1)
\psline(6,3)(6,1)
\psline(6,3)(8,1)
\psline(4,1)(6,-1)
\psline(6,1)(6,-1)
\psline(6,1)(6,-1)
\psline(8,1)(6,-1)

\rput(6,3){$\bullet$}
\put(5.9,3.3){$c\vee f$}
\rput(4,1){$\bullet$}
\put(0,0.9){$b\wedge(c\vee f)$}
\rput(6,1){$\bullet$}
\put(6.4,0.9){$c$}
\rput(8,1){$\bullet$}
\put(8.3,1){$f$}
\rput(6,-1){$\bullet$}
\put(5.9,-1.6){$a$}

\rput(-13,-1){(i)}

\psline(-9,3)(-11,1)
\psline(-9,3)(-7,2)
\psline(-7,2)(-7,0)
\psline(-11,1)(-9,-1)
\psline(-7,0)(-9,-1)

\rput(-9,3){$\bullet$}
\put(-8.9,3.3){$e$}
\rput(-11,1){$\bullet$}
\put(-11.6,0.9){$b$}
\rput(-7,2){$\bullet$}
\put(-6.7,2){$c\vee f$}
\rput(-9,-1){$\bullet$}
\put(-9.1,-1.6){$a$}
\rput(-7,0){$\bullet$}
\put(-6.7,0){$c$}

\rput(2,-1){(ii)}
\end{pspicture}
\end{center}
\caption{}\label{small2}
\end{figure}

We now look at the lattice with the elements $a,b,c,c\vee f,$ and $e.$ Here we have $b\vee(c\vee f)=(b\vee c)\vee f=e\vee f=e$ and 
$b\wedge (c\vee f)\geq a.$ If $b\wedge (c\vee f)=a$ we get again a pentagon sublattice of $L$; see Figure~\ref{small2} (i). Since $L$ is 
modular,  we must have $b\wedge (c\vee f) > a.$ We look at the lattice with elements $a,c,f,b\wedge (c\vee f),$ and $c\vee f$; see Figure~\ref{small2} 
(ii). The following relations hold:
\[
c\wedge(b\wedge (c\vee f))=(c\wedge b)\wedge (c\vee f)=a,
\]
and
\[
c\vee(b\wedge (c\vee f))=(c\vee b)\wedge (c\vee f)=e\wedge (c\vee f)=c\vee f,
\]
the first equality in the latter relation being true by modularity. Moreover, we have
\[
f\wedge (b\wedge(c\vee f))=(f\wedge b)\wedge (c\vee f)=a\wedge (c\vee f)=a,
\]
and 
\[
f\vee (b\wedge (c\vee f))=(f\vee b)\wedge (c\vee f),
\]
again by modularity, and, thus, 
\[
f\vee (b\wedge (c\vee f))\leq c\vee f.
\]
If $f\vee (b\wedge (c\vee f))= c\vee f,$ then we get a diamond sublattice of $L$ as in Figure~\ref{small2} (ii) of smaller rank than $\delta$, which is impossible by our assumption. Hence we must have 
\[
(f\vee b)\wedge (c\vee f) < c\vee f.
\]

\begin{figure}[bht]
\begin{center}
\psset{unit=0.5cm}
\begin{pspicture}(-5.3,-2.5)(4,3.5)

\psline(6,3)(4,1)
\psline(6,3)(6,1)
\psline(6,3)(8,1)
\psline(4,1)(6,-1)
\psline(6,1)(6,-1)
\psline(6,1)(6,-1)
\psline(8,1)(6,-1)

\rput(6,3){$\bullet$}
\put(5.9,3.3){$(c\vee f)\wedge(b\vee f)$}
\rput(4,1){$\bullet$}
\put(0,1.4){$c\wedge(b\vee f)$}
\rput(6,1){$\bullet$}
\put(6.2,0.2){$(c\vee f)\wedge b$}
\rput(8,1){$\bullet$}
\put(8.3,1.5){$f$}
\rput(6,-1){$\bullet$}
\put(5.9,-1.6){$a$}

\rput(-13,-1){(i)}

\psline(-9,3)(-11,1)
\psline(-9,3)(-7,2)
\psline(-7,2)(-7,0)
\psline(-11,1)(-9,-1)
\psline(-7,0)(-9,-1)

\rput(-9,3){$\bullet$}
\put(-8.9,3.3){$c\vee f$}
\rput(-11,1){$\bullet$}
\put(-11.6,0.9){$c$}
\rput(-7,2){$\bullet$}
\put(-6.7,2){$(c\vee f)\wedge(b\vee f)$}
\rput(-9,-1){$\bullet$}
\put(-9.1,-1.6){$a$}
\rput(-7,0){$\bullet$}
\put(-6.7,0){$(c\vee f)\wedge b$}

\rput(2,-1){(ii)}
\end{pspicture}
\end{center}
\caption{}\label{small3}
\end{figure}

Let us consider now the lattice with the elements $a,c, (c\vee f)\wedge b, f\vee(b\wedge(c\vee f))=(c\vee f)\wedge(b\vee f)$, and 
$c\vee f.$ The following equalities hold:
\[
((c\vee f)\wedge b)\wedge c=(c\vee f)\wedge(b\wedge c)=a,
\]
and, by modularity,
\[
c\vee(b\wedge(c\vee f))=(c\vee b)\wedge (c\vee f)=c\vee f.
\]
Next, we have:
\[
c\vee(f\vee(b\wedge(c\vee f)))=(c\vee f)\vee(b\wedge (c\vee f))=c\vee f.
\]
Therefore, if $c\wedge((c\vee f)\wedge(b\vee f))=c\wedge (b\vee f)=a$, then $L$ has a pentagon sublattice; see Figure~\ref{small3} (i). Hence we must 
have \[
c\wedge(b\vee f)>a.
\]
Finally, we look at the lattice with the elements $a,c\wedge(b\vee f), (c\vee f)\wedge b, f,$ and $(c\vee f)\wedge(b\vee f).$ 
The following equalities hold:
\[
f\wedge(c\wedge(b\vee f))=(c\wedge f)\wedge (b\vee f)=a\wedge(b\vee f)=a,
\]
and
\[
f\vee (c\wedge(b\vee f))=(f\vee c)\wedge (b\vee f) \text{ (by modularity) }.
\]
Next,
\[
f\wedge(b\wedge(c\vee f))=(f\wedge b)\wedge (c\vee f)=a\wedge(c\vee f)=a,
\]
and
\[
f\vee(b\wedge(c\vee f))=(f\vee b)\wedge (c\vee f) \text{ (by modularity) }.
\]
We also have:
\[
(c\wedge(b\vee f))\wedge ((c\vee f)\wedge b)=(b\wedge c)\wedge(b\vee f)\wedge (c\vee f)=a
\]
and, by applying modularity,
\[
(c\wedge(b\vee f))\vee ((c\vee f)\wedge b)=((c\wedge(b\vee f))\vee b)\wedge(c\vee f)
\]
\[
=((b\vee c)\wedge(b\vee f))\wedge(c\vee f)=e\wedge(b\vee f)\wedge(c\vee f)=(b\vee f)\wedge(c\vee f).
\]
Consequently, we have got another diamond sublattice of $L$ (see Figure~\ref{small3} (ii)) with a smaller rank than $\delta$, again a contradiction.
\end{proof}

In the proof of the next theorem we use some arguments which are taken from the proof of \cite[Theorem 1.1]{HH1}, but we include them for the convenience of the reader. 

\begin{Theorem}\label{HH conjecture}
Let $L$ be  a modular non-distributive lattice. Then, for any monomial order $<$ on $K[L]$, the initial ideal $\ini_<(I_L)$ is not squarefree.
\end{Theorem}

\begin{proof}
By Lemma~\ref{small diamond}, $L$ has a sublattice $L^\prime$ with $a=\min L^\prime$, $e=\max L^\prime$ such that $\rank e-\rank a=2.$ Let 
$b_1,b_2,\ldots,b_k\in L$, $k\geq 3,$ be the elements of $L$ such that  for any $1\leq i<j\leq n,$ $b_i\vee b_j=e$ and $b_i\wedge b_j=a.$
Therefore, we have the following relations in $I_L:$ $b_ib_j-ae$ for $1\leq i< j\leq k.$

Let $<$ be an arbitrary monomial order on $K[L]$. We may assume that, with respect to this order, we have $b_1>\cdots >b_k.$ We are going to show 
that 
$\ini_<(I_L)$ is not squarefree. We have to analyze the following two cases.

Case 1. Assume that $ae< b_i b_j$ for any $1\leq i< j\leq k.$ Let $b=b_k$ and consider the binomial $f=ab^2e-a^2e^2$ which, by the proof of 
\cite[Theorem 1.1]{HH1}, belongs to $I_L.$ Let us assume that $\ini_<(I_L)$ is squarefree. Then, since $f\in I_L,$ we must have $abe\in \ini_<(I_L)$, 
hence, following the arguments of the proof of \cite[Theorem 1.1]{HH1}, there exists a binomial $g=abe-u\in I_L$ where $u=\ell mn$ with
$\ell, m,n\in L$, all of them in the interval $[a,e]$ of $L,$ and, in addition, with $\ini_<(g)=abe.$ Also, from the arguments of the cited proof, it 
follows that at least two of the variables $\ell, m,n$ are distinct. Indeed, let 
\begin{equation}\label{exprofg}
g=\sum_{i=1}^N x_i(v_i-w_i)
\end{equation}
where each $x_i$ is a variable and
$v_i-w_i$ is a basic binomial of $I_L$ such that $x_1v_1=abe$, $x_iw_i=x_{i+1}v_{i+1}$ for $1\leq i <N,$ and $x_Nw_N=u.$ Then each variable that 
appears in the binomial $x_i(v_i-w_i)$ must belong to the interval $[a,e]$ of $L.$ This is true since for any basic binomial $v-w,$ one has 
$\supp(v)\subset [a,e]$ if and and only if $\supp(w)\subset [a,e].$ In particular, $x_Nw_N=u$ is of the form $u=\ell mn$ with $\ell, m, n\in [a,e]$ and,
by (\ref{exprofg}), at least two of $\ell, m, n$ are distinct.
 Moreover, by (\ref{exprofg}), it also  follows 
that $\rank a+\rank b+\rank e =\rank \ell +\rank m +\rank n.$ Since in $L^\prime$  we have $\rank e-\rank a=2,$ it follows that 
\begin{equation}\label{eqrk}
 \rank \ell +\rank m +\rank n= 3 \rank a +3.
\end{equation} 
Of course we may assume that $\rank \ell \geq \rank m\geq \rank n. $ Let us suppose that $\rank n> \rank a.$ Then, by using equation~(\ref{eqrk}), 
we obtain $\rank \ell =\rank m=\rank n=\rank a +1,$ hence $\ell, m,n \in \{b_1,b_2,\ldots,b_k\}$. It follows that $g=abe-b_ib_jb_p$ for some 
$i,j,p\in \{1,2,\ldots,k\}$ with at least two of them distinct. Let us assume that $i\neq j.$ Then, since $ae< b_ib_j$ and $b\leq b_p$, we get a 
contradiction to the fact that $\ini_<(g)=abe$. 

Let now $\rank n=\rank a.$ This implies that $\rank \ell + \rank m=2 \rank a +3,$ which leads to the conclusion that $\rank \ell=\rank a +2$ and 
$\rank m=\rank a+1.$ Therefore, we get $n=a, \ell=e,$ and $m=b_i$ for some $1\leq i\leq k.$ We then have $g=abe-ab_i e$ which is impossible since 
obviously $abe\leq ab_i e$ by the choice of $b.$

Hence, in  Case 1, $\ini_<(I_L)$ is not squarefree.

Case 2. There exist $1\leq i <j \leq k$ such that $ae >b_ib_j.$ Let $bd$ be the smallest monomial among all the monomials $b_ib_j, 1\leq i < j \leq k.$
In particular, it follows that $ae> bd.$ We first claim that $b^2d-bd^2\in I_L.$ Indeed. one may easily check the following identity:
\[
b^2d-bd^2=(b-d)(bd-ae)-b(cd-ae)+d(bc-ae),
\]
where $c$ is an arbitrary variable in $\{b_1,\ldots,b_k\}\setminus \{b,d\}.$ Let us assume that $\ini_<(I_L)$ is squarefree. Then we have 
$bd\in \ini_<(I_L)$. This implies that there exists a binomial $g=bd-\ell m\in I_L$ with $bd\in\ini_<(I_L).$ Since $ae> bd,$ we cannot have $\ell m=ae.$
Therefore, $g=bd-b_ib_j$ for some $1\leq i< j\leq n,$ which is again impossible by our choice of the monomial $bd$. 
\end{proof}

\section{Radical join-meet ideals of finite lattices} 
\label{radsection}

In this section we describe the associated primes of a radical join-meet ideal of a finite  lattice.

 \begin{Proposition}\label{general}
 Let $S=K[x_1,\ldots,x_n]$ be a polynomial ring over a field $K$ and let $I\subset S$ be a binomial ideal, that is, an ideal which is generated by differences of two monomials. If $I$ is a radical ideal, then:
\begin{itemize}
	\item [(a)] $I: (\prod_{i=1}^n x_i)^\infty=I: \prod_{i=1}^n x_i.$
	\item [(b)] $I: \prod_{i=1}^n x_i$ is a prime ideal.
\end{itemize}
\end{Proposition}
 
\begin{proof}
(a). Let $\Min^\ast(I)$ be the set  of all prime ideals of $I$ which contain no variable.  Then 
\[
I: \prod_{i=1}^n x_i=\bigcap_{P\in \Min(I)}(P:\prod_{i=1}^n x_i)=\bigcap_{P\in \Min^\ast(I)}P=\bigcap_{P\in \Min(I)}(P:(\prod_{i=1}^n x_i)^\infty)=I: (\prod_{i=1}^n x_i)^\infty.
\]

(b). By \cite{ES} or \cite{OP}, $I: \prod_{i=1}^n x_i$ is a lattice ideal, let us say $I_\Lc$ where $\Lc\subset \ZZ^n$ is a lattice. 
By \cite[Theorem 2.1]{ES}, it is enough to show 
that $\Lc$ is saturated, in other words, if $\xb^{m\ab}-\xb^{m\bb}\in I: \prod_{i=1}^n x_i$ for some positive integer $m,$ then 
$\xb^\ab-\xb^\bb\in I: \prod_{i=1}^n x_i.$ 

The proof depends on the characteristic of the field. Let us first assume that $\chara K=0.$ Since, by the proof of (a), we have
$I: \prod_{i=1}^n x_i=\bigcap_{P\in \Min^\ast (I)}P,$ we get $\xb^{m\ab}-\xb^{m\bb}\in P$ for any prime ideal $P\in \Min^\ast(I).$ Since $P$ does not 
contain any variable, it follows that the polynomial $g=\xb^{(m-1)\ab}+\ldots +\xb^{(m-1)\bb}\not \in P$ since $g(1,\ldots,1)=m\neq 0,$ hence 
$\xb^\ab-\xb^\bb\in P$ for any 
$P\in \Min^\ast(I).$ Therefore, we obtain $\xb^\ab-\xb^\bb\in I: \prod_{i=1}^n x_i.$ 

A similar proof works in positive characteristic. Indeed, let $p>0$ be the characteristic of the field and let $m=p^t q$ for some non-negative integer 
$t$ and some positive integer $q$ such that $(p,q)=1$. Then
\[
\xb^{m\ab}-\xb^{m\bb}=(\xb^{q\ab}-\xb^{q\bb})^{p^t}=(\xb^\ab-\xb^\bb)^{p^t}(\xb^{(q-1)\ab}+\cdots + \xb^{(q-1)\bb})^{p^t}\in P
\]
for all $P\in \Min^\ast(I)$. Let $h=(\xb^{(q-1)\ab}+\cdots + \xb^{(q-1)\bb})^{p^t}=\xb^{(q-1)\ab p^t}+\cdots + \xb^{(q-1)\bb p^t}.$ Then $h(1,\ldots,1)=q\neq 0.$ It follows, by using the same argument as in the zero characteristic,  that $(\xb^\ab-\xb^\bb)^{p^t}\in P$ and thus $\xb^\ab-\xb^\bb\in P$ for every $P\in \Min^\ast(I)$. This implies that $\xb^\ab-\xb^\bb\in I: \prod_{i=1}^n x_i.$
\end{proof}

Now we are going to characterize the associated primes of a radical join-meet ideal of a finite lattice. We first need the following

\begin{Definition}
Let $L$ be a lattice and $A$ a subset of $L.$ $A$ is called {\em admissible} if it is empty or it is non-empty and has the following property: for any 
basic binomial $ab-cd$ 
of $I_L$, if $a\in A$ or $b\in A$, then $c\in A$ or $d\in A.$ 
\end{Definition}

In other words, the set $A$ is admissible if and only if, for any basic binomial, either $A$ "covers" both monomials of the binomial or none  of them. 
Of course, the empty set and $L$ are admissible sets for $I_L.$

\begin{Remark}{\em 
Let $A$ be an admissible set for $I_L$. We set  $L_A=L\setminus A.$ Then $L_A$ is a sublattice of $L$ with respect to the order induced from $L.$
Indeed, let $a,b\in L_A$ be two incomparable elements. Since $A$ is admissible, it follows that $a\vee b$ and $a\wedge b$ do not belong to $A.$ 
}
\end{Remark}

\begin{Proposition}\label{radsubset}
Let $I_L$ be a radical ideal. Then, for any admissible set, the ideal $I_{L_A}$ is radical.
\end{Proposition}

\begin{proof}
Assume that there exists $A\subset L$ such that $I_{L_A}$ is not radical, hence there exists a polynomial $f\in K[\{a: a\in L\setminus A\}]$ such 
that  $f\in \sqrt{I_{L_A}}\setminus I_{L_A}.$ Then obviously $f\in \sqrt{I_L}.$ We claim that $f\not\in I_L$ which shows that $I_L$ is not radical, a contradiction. Let us assume that $f\in I_L.$ Then we may write 
\[
f=\sum_{a,b\not\in A}h_{ab}(ab-(a\wedge b)(a\vee b)) + \sum_{a\in A \text{ or }b\in A}h_{ab}(ab-(a\wedge b)(a\vee b))
\]
for some polynomials $h_{ab}\in K[L].$ We map to zero all the variables of $A.$ In this way, since $A$ is admissible, it follows that  the second sum in the above formula vanishes while in the first sum, all the basic binomials survive. Therefore, $f\in I_{L_A},$ a contradiction.
\end{proof}

\begin{Remark}{\em 
We are going to see in Example~\ref{exampleR} that the radical property does 
not pass from a lattice to any of its proper sublattices. 
}
\end{Remark}

For an admissible set $A\subset L,$ we set $$P_A(L)=I_{L_A}:\prod_{a\not\in A}a+ (a : a\in A).$$ If $I_L$ is a radical ideal, then $I_{L_A}$ is a radical ideal by Proposition \ref{radsubset}, and, by Proposition~\ref{general}, it follows that $I_{L_A}:\prod_{a\not\in A}a$ is prime. Thus $P_A(L)$ is a prime ideal for any admissible set $A$ if $I_L$ is a radical ideal. Obviously, $P_A(L)\supset I_L$ for any admissible set $A.$

 \begin{Theorem}\label{intersection}
 Let $L$ be a lattice such that $I_L$ is a radical ideal. Then 
 \[
 I_L=\bigcap\limits_{A\subset L\atop A\text{ admissible }}P_A(L).
 \] 
 \end{Theorem}
 
\begin{proof}
It is enough to show that any minimal prime ideal of $I_L$ is of the form $P_A(L)$ for some admissible set $A\subset L.$ 

Let $P$ be a minimal prime of 
$I_L$ and $A=\{a: a\in P\}$. If $A=\emptyset,$ that is, $P$ does not contain any variable, then $P\supset I_L:\prod_{a\in L}a\supset I_L.$ 
Since, by Proposition~\ref{general}, $I_L:\prod_{a\in L}a$ is a prime ideal, we obtain $P=P_\emptyset(L).$ 

Now let $A$ be nonempty. We claim that $A$ is admissible. Indeed, let $ab-cd$ be a basic binomial such that $a\in A.$ It follows that 
$cd \in P$, which implies that $c\in A$ or $d\in A.$ We show that $P=P_A(L).$ Indeed, since $P\supset I_L$ and $P\supset (a: a\in A)$, we also 
have $P\supset I_L+(a:a\in A)=I_{L_A}+(a:a\in A).$ It follows that $P\supset (I_{L_A}+(a:a\in A)):\prod_{a\notin A}a=P_A(L).$ Since $P$ is 
minimal over $I$, we must have $P=P_A(L).$
\end{proof}

\begin{Proposition}\label{mincond}
Let $I_L$ be radical. Then for two admissible sets $A,B\subset L$, we have $P_A(L)\subsetneq P_B(L)$ if and only if 
\[
A\subsetneq B
\text{ and } I_{L_A}:\prod_{a\not\in A}a\subset I_{L_B}:\prod_{b\not\in B}b +(b: b\in B\setminus A). 
\]
\end{Proposition}

\begin{proof}
Let $A\subset B$. Then $P_A(L)\subset P_B(L)$ if and only if $$I_{L_A}:\prod_{a\not\in A}a=P_A(L)/ (a: a\in A)\subset P_B(L)/(a: a\in A)$$ 
$$=I_{L_B}:\prod_{b\not\in B}b +(b: b\in B\setminus A).$$
\end{proof}
 
The following example illustrates Theorem~\ref{intersection} and Proposition~\ref{mincond}.

\begin{Example}{\em 
Let $Q$ be the lattice of Figure~\ref{latticeQ}. The Gr\"obner basis of $I_Q$ with respect to the lexicographic order induced by $a>b>\cdots> g$ is 
$\{ae-bc,ag-cf,bg-ef,cd-cf,de-ef\}.$ Thus, $\ini_<(I_Q)$ is squarefree which implies that $I_Q$ is a  radical ideal and we may apply Theorem~\ref{intersection} and Proposition~\ref{mincond} to determine the minimal primes of $I_Q.$

\begin{figure}[hbt]
\begin{center}
\psset{unit=0.5cm}
\begin{pspicture}(-1.3,-5.5)(1,2)

\psline(0,-5)(2,-3)
\psline(0,-5)(-2,-3)
\psline(2,-3)(0,-1)
\psline(-2,-3)(-3,-2)
\psline(0,-1)(-1.5,1.5)
\psline(-3,0)(-1.5,1.5)
\psline(-2,-3)(0,-1)
\psline(-3,-2)(-3,0)

\rput(0,-5){$\bullet$}
\put(-0.8,-5){$a$}
\rput(2,-3){$\bullet$}
\put(2.3,-3){$c$}
\rput(-2,-3){$\bullet$}
\put(-2.8,-3.2){$b$}
\rput(0,-1){$\bullet$}
\put(0.3,-1){$e$}
\rput(-3,-2){$\bullet$}
\put(-3.8,-2){$d$}
\rput(-3,0){$\bullet$}
\put(-3.8,0){$f$}
\rput(-1.5,1.5){$\bullet$}
\put(-2.3,1.5){$g$}

\end{pspicture}
\end{center}
\caption{}\label{latticeQ}
\end{figure}
 One easily sees that
\[
I_Q:\prod_{x\in Q}x\supset J=(ae-bc,ag-cf,bg-ef,d-f)\supset I_Q.
\]
 But $K[a,b,c,d,e,f,g]/J\cong K[a,b,c,d,e,g]/(ae-bc,ag-cd,bg-de)$, and the latter 
quotient ring is a domain. Therefore, $J$ is a prime ideal. Moreover, $I_Q:\prod_{x\in Q}x=J=P_\emptyset (Q).$ The other minimal primes 
of $I$ are $(a,b,c,e)$ and $(c,e,g)$, that is, $I=J\cap(a,b,c,e)\cap(c,e,g).$ Note that, for instance, the set $A=\{g,d,f\}$ is an admissible set, 
but the corresponding prime ideal $P_A(Q)$ is not a minimal prime of $I_Q$ since $P_A(Q)\supsetneq P_\emptyset(Q).$
}
\end{Example}
 
\section{Join-meet ideals of modular non-distributive lattices }
\label{modular}

It is well known that, given an ideal $I$ of a polynomial ring $S$ over a field, if $\ini_<(I)$ is radical for some monomial order $<$ on $S,$ then 
the ideal $I$ is radical as well; see \cite[Proposition 3.3.7]{HHbook} or \cite[Lemma 6.51]{EH} for an alternative proof. This gives also a procedure 
to show that a polynomial ideal is radical. However, there are radical polynomial ideals whose initial ideals are always non-radical. For such ideals 
one has to use other kind of arguments to prove the radical property. 

In this section  we mainly study a class of modular non-distributive lattices whose join-meet ideals are radical. Before beginning our study, let us look at the next

\begin{Example}\label{latticeN}{\em 

Let $N$ be the lattice of rank $4$ of Figure~\ref{FigN}. This is rather a simple example of a modular non-distributive lattice. We "included" only 
one diamond into a distributive lattice with $8$ elements. However, as we are going to show, the join-meet ideal of lattice $N$ is not radical.

\begin{figure}[hbt]
\begin{center}
\psset{unit=0.5cm}
\begin{pspicture}(-25.3,-2.5)(4,5)

\psline(-9,3)(-11,1)
\psline(-9,3)(-9,1)
\psline(-9,3)(-7,1)
\psline(-11,1)(-9,-1)
\psline(-9,1)(-9,-1)
\psline(-7,1)(-9,-1)
\psline(-9,3)(-11,5)
\psline(-13,3)(-11,5)
\psline(-13,3)(-11,1)
\psline(-11,1)(-13,-1)
\psline(-13,-1)(-11,-3)
\psline(-11,-3)(-9,-1)

\rput(-9,3){$\bullet$}
\put(-9.1,3.3){$h$}
\rput(-11,1){$\bullet$}
\put(-11.6,0.9){$d$}
\rput(-9,1){$\bullet$}
\put(-8.6,0.9){$e$}
\rput(-7,1){$\bullet$}
\put(-6.7,1){$f$}
\rput(-9,-1){$\bullet$}
\put(-9.1,-1.6){$c$}
\rput(-11,-3){$\bullet$}
\put(-11,-3.5){$a$}
\rput(-13,-1){$\bullet$}
\put(-13.5,-1){$b$}
\rput(-13,3){$\bullet$}
\put(-13.5,3){$g$}
\rput(-11,5){$\bullet$}
\put(-11,5.5){$\ell$}
\end{pspicture}
\end{center}
\caption{}\label{FigN}
\end{figure}
We claim that $a\ell g(d-f)^2\in I_L,$ which implies that $(a\ell g(d-f))^2\in I_L,$ therefore, $a\ell g(d-f)\in \sqrt{I_L}.$ Indeed, one may easily see that 
\[
a\ell g(d-f)^2=a\ell gd^2-2a\ell gdf + a \ell gf^2 \equiv ag^2hd-ag^2h f -agf(gh-\ell f)
\]
\[
\equiv ag^2h(d-f)-agf\ell ( d- f)\equiv ag^2h(d-f)-a\ell^2 c(d-f) \mod I_L.
\]
On the other hand, $ah(d-f)\in I_L$ and $\ell c(d-f)\in I_L$. One may easily check this. For instance, for the first membership, we may use the 
following identity:
\[
ah(d-f)=b(de-ch)+(f-d)(be- ah)-b(ef-ch).
\]
Thus, $a\ell g(d-f)\in \sqrt{I_L}.$ The Gr\"obner basis of $I_L$ with respect to reverse 
lexicographic order contains, apart of the basic binomials of $L$, the following binomials: $ce\ell-cf\ell, cd\ell-cf\ell,ceh-cfh,aeh-afh,
cdh-cfh,adh-afh,cf^2\ell-c^2h\ell,ad^2\ell-ach\ell,cf^2h-c^2h^2,af^2h-ach^2.$ Thus, $\ini_<(a\ell gd-a\ell gf)\not\in \ini_<(I_L)$ which implies that 
$a\ell g(d-f)\not\in I_L$.
}
\end{Example}

Therefore, the following question arises. Is there a class of distributive lattices such that by "including" just one small diamond one may 
get a radical joint-meet ideal for the new lattice? We are going to answer this question in the next theorem. 

Let $D$ be the distributive lattice of the divisors of $2\cdot 3^n$ for some integer $n\geq 1$ with the elements labeled as in Figure~\ref{chain} (a). For every $1\leq k\leq n-1,$ we denote by $L_k$ the lattice of Figure~\ref{chain} (b). 

\begin{figure}[hbt]
\begin{center}
\psset{unit=0.5cm}
\begin{pspicture}(-1,-11)(5,2)

\psline(-9,1)(-11,-1)
\psline(-9,1)(-7,-1)
\psline(-9,-3)(-7,-1)
\psline(-11,-1)(-9,-3)
\psline(-5,-3)(-7,-5)
\psline(-7,-5)(-5,-7)
\psline(-5,-7)(-3,-5)
\psline(-3,-5)(-5,-3)
\psline(-3,-9)(-1,-7)
\psline(-1,-7)(1,-9)
\psline(1,-9)(-1,-11)
\psline(-1,-11)(-3,-9)
\psline[linestyle=dotted](-9,-3)(-7,-5)
\psline[linestyle=dotted](-7,-1)(-5,-3)
\psline[linestyle=dotted](-5,-7)(-3,-9)
\psline[linestyle=dotted](-3,-5)(-1,-7)

\rput(-9,1){$\bullet$}
\put(-8.5,1){$y_n$}
\rput(-11,-1){$\bullet$}
\put(-12.5,-1){$x_n$}
\rput(-7,-1){$\bullet$}
\put(-6.5,-1){$y_{n-1}$}
\rput(-9,-3){$\bullet$}
\put(-11,-3){$x_{n-1}$}
\rput(-5,-3){$\bullet$}
\put(-4.5,-3){$y_{k+1}$}
\rput(-7,-5){$\bullet$}
\put(-9,-5){$x_{k+1}$}
\rput(-3,-5){$\bullet$}
\put(-2.5,-5){$y_k$}
\rput(-5,-7){$\bullet$}
\put(-6.5,-7){$x_k$}
\rput(-1,-7){$\bullet$}
\put(-0.5,-7){$y_2$}
\rput(-3,-9){$\bullet$}
\put(-4.5,-9){$x_2$}
\rput(1,-9){$\bullet$}
\put(1.5,-9){$y_1$}
\rput(-1,-11){$\bullet$}
\put(-2.5,-11){$x_1$}

\rput(-7,-9){(a)}

\psline(6,1)(4,-1)
\psline(6,1)(8,-1)
\psline(6,-3)(8,-1)
\psline(4,-1)(6,-3)
\psline(10,-3)(8,-5)
\psline(8,-5)(10,-7)
\psline(10,-7)(12,-5)
\psline(12,-5)(10,-3)
\psline(12,-9)(14,-7)
\psline(14,-7)(16,-9)
\psline(16,-9)(14,-11)
\psline(14,-11)(12,-9)
\psline[linestyle=dotted](6,-3)(8,-5)
\psline[linestyle=dotted](8,-1)(10,-3)
\psline[linestyle=dotted](10,-7)(12,-9)
\psline[linestyle=dotted](12,-5)(14,-7)
\psline(10,-3)(10,-5)
\psline(10,-5)(10,-7)

\rput(6,1){$\bullet$}
\put(6.5,1){$y_n$}
\rput(4,-1){$\bullet$}
\put(2.5,-1){$x_n$}
\rput(8,-1){$\bullet$}
\put(8.5,-1){$y_{n-1}$}
\rput(6,-3){$\bullet$}
\put(4,-3){$x_{n-1}$}
\rput(10,-3){$\bullet$}
\put(10.5,-3){$y_{k+1}$}
\rput(8,-5){$\bullet$}
\put(6,-5){$x_{k+1}$}
\rput(12,-5){$\bullet$}
\put(12.5,-5){$y_k$}
\rput(10,-7){$\bullet$}
\put(8.5,-7){$x_k$}
\rput(14,-7){$\bullet$}
\put(14.5,-7){$y_2$}
\rput(12,-9){$\bullet$}
\put(10.5,-9){$x_2$}
\rput(16,-9){$\bullet$}
\put(16.5,-9){$y_1$}
\rput(14,-11){$\bullet$}
\put(12.5,-11){$x_1$}
\rput(10,-5){$\bullet$}
\put(10.3,-5){$z$}

\rput(8,-9){(b)}

\end{pspicture}
\end{center}
\caption{}\label{chain}
\end{figure}

Before stating our first preparatory result, we need to introduce some notation. For $1\leq k\leq n-1,$ let 
\[
p_k=x_{k+1}z-y_k z; r_k=y_k^2z-y_kz^2; g_i=x_iy_{k+1}-y_iz, \text{ for }1\leq i <k;
\]
\[
 h_j=x_ky_j-x_jz, \text{ for }k+1\leq j \leq n;
f_{ij}=x_jy_i-x_iy_j, \text{ for } 1\leq i < j\leq n, j\neq k+1, i\neq k; 
\]
\[
f_{i,k+1}=x_{k+1}y_i-y_iz, \text{ for }1\leq i\leq k; f_{kj}=x_jy_k-x_jz, \text{ for } j> k+1; 
\]
\[
p_{ij}=x_ix_{k+1}y_j-x_iy_jz, t_{ij}=x_iy_ky_j-x_iy_jz, \text{ for } 1\leq i< k< k+1 < j\leq n,
\]
and
\[
q_{ik}= y_iy_kz-y_iz^2, \text{ for }1\leq i <k.
\]

\begin{Lemma}\label{GBofI}
The set 
\[
\Gc=\{p_k,r_k\}\cup \{g_i,q_{ik} :  1\leq i<k\}\cup \{h_j: k+1\leq j\leq n\}\cup\{f_{ij}: 1\leq i < j\leq n\}
\]
\[
\cup \{p_{ij}, t_{ij}: 1\leq i<k<k+1 < j\leq n\}
\]
is a Gr\"obner basis of $I=I_{L_k}$ with respect to the reverse lexicographic order induced by $x_1>\cdots >x_n> y_1>\cdots> y_n>z.$
In particular, it follows that $\ini_<(I)$ is generated by the following set of monomials: 
\[
\Mc=\{x_jy_i: 1\leq i< j\leq n\}\cup\{x_iy_{k+1}: 1\leq i< k\}\cup\{x_ky_j: k+1\leq j\leq n\}
\]
\[
\cup\{x_ix_{k+1}y_j,x_iy_ky_j: 1\leq i< k< k+1 < j\leq n\}\cup\{y_iy_kz: 1\leq i<k\}\cup\{x_{k+1}z,y_k^2z\}.
\]
\end{Lemma}

\begin{proof}
We first note that $\Gc$ is a generating set of $I$ and next one applies Buchberger's criterion, that is, one checks that all the $S$-polynomials 
of the pairs $(f,g)\in \Gc\times \Gc$ reduce to zero modulo $\Gc.$ Note that for many pairs $(f,g)\in \Gc\times \Gc$ the checks are superfluous since 
the initial monomials $\ini_<(f)$ and $\ini_<(g)$ are relatively prime. Moreover, in order to eliminate many checks, one may use the following known 
fact.  If $f,g$ are two polynomials with $\ini_<(f)$ and $\ini_<(g)$ relatively prime, then, for any monomials $u,v$ the $S$-polynomial $S(uf,vg)$ 
reduces to zero modulo $uf$ and $vg.$
\end{proof}

\begin{Theorem}\label{radical}
For every $1\leq k\leq n-1,$ the join-meet ideal $I_{L_k}$ is radical.
\end{Theorem}

The proof of this theorem has  several steps which are shown in the following lemmas, but the basic idea of the proof is very simple. We 
actually show that one may decompose $I$ as an  intersection of two  radical ideals, namely $I=(I,x_{k+1}-y_k)\cap (I,z),$ hence $I$ itself is a 
radical ideal.

\begin{Lemma}\label{intersect}
Let $1\leq k\leq n-1 $ and $I=I_{L_k}.$ Then $I=(I,x_{k+1}-y_k)\cap (I,z).$
\end{Lemma}

\begin{proof}
The inclusion $I\subset (I,x_{k+1}-y_k)\cap (I,z)$ is obvious. For getting the equality we show that 
\begin{equation}\label{eqint}
\ini_<(I,x_{k+1}-y_k)\cap \ini_<(I,z)\subset \ini_<(I).
\end{equation}
This will imply that $\ini_<((I,x_{k+1}-y_k)\cap (I,z))\subset \ini_<(I),$ thus, 
\[
\ini_<(I)=\ini_<((I,x_{k+1}-y_k)\cap (I,z))
\]
which leads to the desired statement.

We know the generators of $\ini_<(I)$ from Lemma~\ref{GBofI}. We now compute the Gr\"obner bases of $(I,z)$ and $(I,x_{k+1}-y_k)$ with respect to the 
reverse lexicographic order induced by $x_1>\cdots >x_n> y_1>\cdots >y_n>z$. 
By using the Gr\"obner basis of $I,$ one easily sees that $(I,z)$ is generated by the binomials $f_{ij}=x_jy_i-x_iy_j$ where $1\leq i<j\leq n$ and 
$j\neq k+1$, $i\neq k$ and by the following set of monomials: $\{z\}\cup\{x_iy_{k+1}: 1\leq i<k\}\cup\{x_ky_j: k+1\leq j\leq n\}\cup\{x_{k+1}y_i: 
1\leq i\leq k\}\cup\{x_jy_k: j> k+1\}\cup \{x_ix_{k+1}y_j, x_iy_ky_j: 1\leq i<k<k+1<j\leq n\}.$ By using Buchberger's criterion, one immediately 
checks that the above set of generators of $(I,z)$ is a Gr\"obner basis of $(I,z)$. Consequently, 
\[
G(\ini_<(I,z))=(G(\ini_<(I)\setminus\{x_{k+1}z,y_kz^2,y_iy_kz:1\leq i<k\})\cup\{z\} 
\]
which implies that 
\begin{equation}\label{eqini1}
\ini_<(I,z)=(\ini_<(I),z).
\end{equation}
Here we used the notation $G(J)$ for the minimal set of  monomial generators of the monomial ideal $J.$

By using the Gr\"obner basis of $I$ it follows that the ideal $(I,x_{k+1}-y_k)$ is generated by the binomials $x_{k+1}-y_k, g_i, 1\leq i<k, h_j, k+1\leq j\leq n, 
f_{ij}, 1\leq i<j\leq n, j\neq k+1,i\neq k, 
f_{i,k+1}^\prime=y_iy_k-y_iz=f_{i,k+1}-y_i(x_{k+1}-z), 1\leq i\leq k, f_{kj}, j> k+1,r_k,$ and $p_{ij}^\prime=t_{ij}=x_iy_jy_k-x_iy_jz, 1\leq i<k<k+1<j\leq n,$ since $q_{ik}=zf_{i,k+1}^\prime$.  Buchberger's criterion applied to this set of generators shows that they form a Gr\"obner basis of $(I,x_{k+1}-y_k).$ Moreover, we obtain
\[
G(\ini_<(I,x_{k+1}-y_k))=(G(\ini_<(I)\setminus(\{x_{k+1}z, x_{k+1}y_i: 1\leq i\leq k \}\cup\{y_iy_kz: 1\leq i<k\}))
\]
\[
\cup\{x_{k+1},y_iy_k: 1\leq i\leq k\}.
\]
therefore, we get the following equality:
\begin{equation}\label{eqini2}
\ini_<(I,x_{k+1}-y_k)=(\ini_<(I),x_{k+1},y_1y_k,\ldots,y_{k-1}y_k,y_k^2).
\end{equation}

By using the relations (\ref{eqini1}) and (\ref{eqini2}), we get 
\[
\ini_<(I,x_{k+1}-y_k)\cap \ini_<(I,z)=(\ini_<(I),x_{k+1}z, y_1y_kz,\ldots,y_k^2z)\subset \ini_<(I).
\]
\end{proof}

From the above proof we may also derive the following 

\begin{Corollary}\label{corsquarefree} 
$(I,z)$ is a radical ideal.
\end{Corollary}

\begin{proof}
By (\ref{eqini1}), we have $\ini_<(I,z)=(\ini_<(I),z)$. Since $\ini_<(I)$ has only one non-squarefree generator, namely $y_k^2z$ which is "killed" by 
$z,$ it follows that $\ini_<(I,z)$ is square free and, consequently, $(I,z)$ is a radical ideal.
\end{proof}

The last step in the proof of Theorem~\ref{radical} is shown in the following 

\begin{Lemma}\label{sqfree2}
The ideal $(I,x_{k+1}-y_k)$  is  radical.
\end{Lemma}

\begin{proof}
We show that $(I,x_{k+1}-y_k)$  has a squarefree initial ideal with respect to the lexicographic order induced by 
$z> x_1>\cdots >x_n>y_1>\cdots >y_n.$ We recall from the proof of Lemma~\ref{intersect} that $(I,x_{k+1}-y_k)$ is generated by 
$x_{k+1}-y_k, g_i, 1\leq i<k, h_j, k+1\leq j\leq n,  f_{ij}, 1\leq i<j\leq n, j\neq k+1,i\neq k, 
f_{i,k+1}^\prime=y_iy_k-y_iz, 1\leq i\leq k, f_{kj}, j> k+1,r_k,$ and $p_{ij}^\prime=t_{ij}=x_iy_jy_k-x_iy_jz, 1\leq i<k<k+1<j\leq n.$ In this 
generating set, the generators $r_k$ and $p_{ij}^\prime$ are redundant. Indeed, $r_k=z f_{k,k+1}^\prime$ and $p_{ij}^\prime=(y_k-z)f_{ij}
-x_j f_{i,k+1}^\prime$ for any $1\leq i< k< k+1 < j\leq n.$ Moreover, for every $1\leq i<k$ we may replace the generator $g_i$ by
$g_i^\prime=x_iy_{k+1}-y_iy_k=f_{i,k+1}^\prime- g_i.$ Finally, for $j> k+1$ we may replace the generator $f_{kj}$ by $x_ky_j-x_jy_k=f_{kj}-h_j.$ 
Therefore, $(I,x_{k+1}-y_k)$ is generated by the following binomials: $x_{k+1}-y_k, g_i^\prime=x_iy_{k+1}-y_iy_k$ for $1\leq i<k,$ $h_j=zx_j-x_ky_j$
for $k+1\leq j\leq n$, $f_{i,k+1}^\prime=zy_i-y_iy_k$ for $1\leq i\leq k$, and $f_{ij}=x_iy_j-x_jy_i$ for $1\leq i< j\leq n$ with $j\neq k+1.$ By trivial 
calculations  one may check that this set of generators is a Gr\"obner basis of $(I,x_{k+1}-y_k)$ with respect to the lexicographic order induced by 
$z>x_1>\cdots >x_n>y_1>\cdots >y_n.$ Since all these generators have squarefree initial monomials, it follows that the initial ideal of $(I,x_{k+1}-y_k)$ is squarefree and, thus, $(I,x_{k+1}-y_k)$ is a radical ideal. 
\end{proof}

We end this section with a few comments. Going back to our Example~\ref{latticeN}, by applying  Theorem~\ref{radical}, we see that every proper 
sublattice $N^\prime$ of $N$ has a radical join-meet ideal although $I_N$ is not radical. The following example shows that the radical property does 
not pass from a lattice to any of its proper sublattices. 

\begin{Example}\label{exampleR} {\em
Let $R$ be the lattice of Figure~\ref{latticeR}.

\begin{figure}[hbt]
\begin{center}
\psset{unit=0.5cm}
\begin{pspicture}(-25.3,-2.5)(4,5)

\psline(-9,3)(-11,1)
\psline(-9,3)(-9,1)
\psline(-9,3)(-7,1)
\psline(-11,1)(-9,-1)
\psline(-9,1)(-9,-1)
\psline(-7,1)(-9,-1)
\psline(-9,3)(-11,5)
\psline(-13,3)(-11,5)
\psline(-13,3)(-11,1)
\psline(-11,1)(-13,-1)
\psline(-13,-1)(-11,-3)
\psline(-11,-3)(-9,-1)
\psline(-15,1)(-13,-1)
\psline(-15,1)(-13,3)

\rput(-9,3){$\bullet$}
\rput(-11,1){$\bullet$}
\rput(-9,1){$\bullet$}
\rput(-7,1){$\bullet$}
\rput(-9,-1){$\bullet$}
\rput(-11,-3){$\bullet$}
\rput(-13,-1){$\bullet$}
\rput(-13,3){$\bullet$}
\rput(-11,5){$\bullet$}
\rput(-15,1){$\bullet$}
\end{pspicture}
\end{center}
\caption{}\label{latticeR}
\end{figure}
One may check with Singular \cite{GPS} that $I_R$ is a radical ideal. However the ideal $I_N$ attached to  its proper sublattice $N$ is not radical, as we have seen in Example~\ref{latticeN}.
}
\end{Example}

\section{The minimal primes of the join-meet ideal of $L_k$}
\label{minsection}

In this section we apply the results of Section~\ref{radsection} to determine explicitly the minimal primes of the ideals $I_{L_k}$ for 
$1\leq k\leq n-1.$ We recall that we denoted by $D$ the distributive lattice displayed in Figure~\ref{chain} (a), and by $L_k$ the 
lattice displayed in Figure~\ref{chain} (b). We denote by $D_k$ the sublattice of $D$ with the elements $x_i,y_i, 1\leq i\leq k,$ and by $D_k^\prime$ the sublattice of $D$ with the elements $x_{i},y_i, k+1\leq i\leq n.$

Before stating the main theorem of this section, we need to prove a preparatory result.

\begin{Lemma}\label{ufff}
For any $1\leq k\leq n-1,$ the ideal $(I_D,x_{k+1}-y_k)$ is prime.
\end{Lemma}

\begin{proof}
It is enogh to show that $(I_D,x_2-y_1)$ is a prime ideal since by an appropriate change of variables, we may map the ideal $(I_D,x_2-y_1)$ into 
$(I_D,x_{k+1}-y_k).$

Let $f_{ij}=x_iy_j-x_jy_i,$ $1\leq i < j\leq n$ the generators of $I_D.$ By \cite[Theorem 2.2]{HH1}, $\{f_{ij}: 1\leq i< j\leq n\}$ is a Gr\"obner 
basis of $I_D$ with respect to any monomial order. Actually, if $\ini_<f_{ij}$ and $\ini_<f_{k\ell}$ are not relatively prime, then the 
$S$-polynomial of the pair $(f_{ij},f_{k\ell})$ may be expressed as 
\begin{equation}\label{spoly}
S(f_{ij},f_{k\ell})=z f_{pq}
\end{equation} 
for some variable $z\in K[D]$ and $1\leq p< q\leq n.$

Let $<$ be an arbitrary monomial order on $K[D].$ For any $1\leq i < j\leq n,$ we denote by $g_{ij}$ the reduction of $f_{ij}$ modulo $x_2-y_1$. More 
precisely, $g_{ij}$ is obtained from $f_{ij}$ by replacing $x_2$ by $y_1$ if $x_2>y_1$ or $y_1 $ by $x_2$ if $y_1> x_2.$ Since 
$\{f_{ij}:1\leq i < j\leq n\}$ is a Gr\"obner basis of $I_D$ with respect to $<$, it follows that the set $\Gc=\{g_{ij}:1\leq i < j\leq 
n\}\cup\{x_2-y_1\}$ is a Gr\"obner basis of $(I_D,x_2-y_1)$ with respect to $<.$ This is essentially due to equation (\ref{spoly}). In particular, 
$\Gc$ is a Gr\"obner basis of $(I_D,x_2-y_1)$ with respect to the lexicographic order induced by $x_1>\cdots > x_n> y_1 >\cdots > y_n$. In this case 
it follows that the initial ideal of $(I_D,x_2-y_1)$ is generated by the following squarefree monomials: $x_2,x_iy_j$ for $i,j\neq 2$, $x_1y_2$, and 
$x_jy_2$ for $2<j\leq n.$ This shows that $(I_D,x_2-y_1)$ is a radical ideal. 
On the other hand, by applying \cite[Lemma 12.1]{S}, it follows that all the variables are regular on $(I_D,x_2-y_1)$, which implies that 
$(I_D,x_2-y_1): \prod_{1\leq i\leq n}x_i\prod_{1\leq j\leq n}y_j=(I_D,x_2-y_1).$ Finally, by applying Proposition~\ref{general}, we get the desired conclusion.
\end{proof}

\begin{Theorem}\label{minprimes}
Let $1\leq k\leq n-1$ and $I=I_{L_k}$ the join-meet ideal of the lattice $L_k.$ The minimal primes of $I$ are the followings:
\[
P=(I, z-x_{k+1},z-y_k), P_1=(z,x_1,\ldots,x_n), P_1^\prime=(z,y_1,\ldots,y_n),
\]
\[
P_2=(z,x_1,\ldots,x_k,y_1,\ldots,y_k)+I_{D_k^\prime}, P_2^\prime=(z,x_{k+1},\ldots,x_n, y_{k+1},\ldots,y_n) + I_{D_k},
\]
\[
P_3=(x_1,\ldots,x_n,y_1,\ldots,y_k), P_3^\prime=(y_1,\ldots,y_n,x_{k+1},\ldots,x_n).
\]
\end{Theorem}

\begin{proof}
By Theorem~\ref{intersection}, since $I$ is a radical ideal,  we know that any minimal prime of $I$ is of the form $P_A(L_k)$ where 
$A$ is an admissible set of $I.$ 

Let $P=P_{\emptyset}(L_k).$ Then $P=I: (z\prod_{1\leq i\leq n}x_i\prod_{1\leq j\leq n}y_j).$ We obviously have
\begin{equation}\label{eqstar}
P\supset (I,z-x_{k+1},x-y_k)\supset I.
\end{equation} 

On the other hand, 
\[
\frac{K[L_k]}{(I,z-x_{k+1},x-y_k)}\cong \frac{K[D]}{(I_D,x_{k+1}-y_k)}.
\]
Since, by Lemma~\ref{ufff}, $(I_D,x_{k+1}-y_k)$ is a prime ideal, it follows that $(I,z-x_{k+1},x-y_k)$ is a prime ideal as well. Therefore, since $P$ 
is a minimal prime of $I,$ by using (\ref{eqstar}), we must have $P=(I,z-x_{k+1},x-y_k).$

Now we look at the minimal primes which correspond to non-empty admissible sets. Let $A$ be such an admissible set and assume first that $z\in A.$ 
If $y_\ell\not\in A$ for every $1\leq \ell\leq n,$ then, by using the basic binomials $zy_i-x_iy_{k+1}$ for $i\leq k$ and $x_ky_j-x_jy_k$ for $j\geq 
k+1$, it follows that $P\supset (z,x_1,\ldots,x_n)\supset I,$ hence, we get $P_A(L_k)=(z,x_1,\ldots,x_n)=P_1.$ Since the dual lattice of $L_k$ has 
obviously the same relation ideal, it follows that $P_1^\prime$ is the minimal prime which correspond to the admissible set $A$ which contains $z$ 
and does not contain any of the variables $x_i, i=1,\ldots,n$. 

Now we consider an admissible set $A$ which contains $z$ and has the property that there exist $1\leq i,j\leq n$ such that  $x_i,y_j\not\in A.$ If 
$i\neq j,$ then, since $x_iy_j-x_jy_i$ is a basic binomial, it follows that $x_j,y_i\not\in A.$ Therefore, we may assume that there exists 
$1\leq i\leq n$ such that $x_i,y_i\not\in A.$ Let us suppose that $x_k,y_k\not\in A$. From the relations $x_jz-x_ky_j$ we get $y_j\in A$ for $j\geq 
k+1$ 
and, next, from the relations $x_jy_k-x_ky_j$, we get  $x_j\in A$ for $j\geq k+ 1$. Thus, in this case, $P_A(L_k)\supset P_2^\prime\supset I$. But 
$P_2^\prime$ is 
obviously a prime ideal, therefore, $P_A(L)=P_2^\prime.$ The dual situation correspond to $x_{k+1},y_{k+1}\not\in A,$ and in this case one gets 
$P_A(L_k)=P_2.$ It remains to consider $x_k,y_k,x_{k+1},y_{k+1}\in A.$ Then it follows that $P_A(L_k)\supsetneq P$ which implies that $P_A(L_k)$ is 
not a minimal prime.

We still need to identify the minimal primes which correspond to non-empty admissible sets $A$ which do not contain $z.$ Let $A$ be such that 
$z\not\in A$ and $P_A(L_k)$ is a minimal prime of $I.$ Since $zy_k-x_ky_{k+1}, zx_{k+1}-x_ky_{k+1}, y_kx_{k+1}-x_ky_{k+1}\in I\subset P_A(L_k)$, we 
get $z(y_k-x_{k+1})\in P_A(L_k),$ hence $y_k-x_{k+1}\in P_A(L)$, and $x_{k+1}(z-y_k)\in P_A(L)$. If $x_{k+1}\not\in P_A(L)$, it follows that 
$z-x_{k+1}\in P_A(L).$ But this further implies that $P_A(L_k)\supsetneq P,$ hence $P_A(L_k)$ is not a minimal prime. Consequently, 
$x_{k+1}\in A,$ and, next, $y_k\in P_A(L_k).$ By using again the basic binomial $ y_kx_{k+1}-x_ky_{k+1},$ we obtain $x_k\in A$ or $y_{k+1}\in A.$

We analyze the following cases. 

Case 1. $x_k\in A$ and $y_{k+1}\not\in A.$ By using the relations $x_jz-x_ky_j$ for $j> k+1$, we get $x_j\in A$ for $j>k+1.$ Similarly, by using 
the basic binomials $y_{k+1}x_i-y_ix_{k+1}$ for $i< k,$ we get $x_i\in A$ for all $i<k.$ Therefore, we have $x_i\in A$ for all $i=1,\ldots,n$. By 
using the basic binomials $zy_i-x_iy_{k+1}$ for $i< k,$ we also get $y_i\in A$. Then we have actually proved that $P_A(L_k)\supset (x_1,\ldots,x_n,y_1,\ldots,y_k)=P_3\supset I$. Since $P_A(L_k)$ is a minimal prime of $I,$ we must have $P_A(L_k)=P_3.$

Case 2. $x_k\not\in A$ and $y_{k+1}\in A.$ This is the dual of the above case and leads to the conclusion that $P_A(L_k)=P_3^\prime.$ 

Case 3. Let $x_k,y_{k+1}\in A.$ From the relations $zy_i-x_iy_{k+1}$ for $i<k,$ and $x_jz-x_ky_j$ for $j> k$, we obtain 
$y_i\in A$ for $i<k$, and $x_j\in A$ for $j>k.$ If there exists $i<k$ such that $x_i\not\in A,$ by using the relations $x_iy_j-x_jy_i$ for $j>k+1$, we get $y_j\in A$ for all $j>k+1.$ In this case it follows that $A\supset \{y_1,\ldots,y_n,x_k,\ldots,x_n\}$ and $P_A(L_k)\supsetneq P_3^\prime,$ hence $P_A(L_k)$ is not a minimal prime, contradiction. In other words, Case 3 does not hold, and this completes the proof.
\end{proof}

\begin{Corollary}\label{dim}
The join-meet ideal$I_{L_k}$ is not unmixed and $\dim(K[L_k]/I_{L_k})=n.$
\end{Corollary}

\begin{proof}
It is known (see \cite{H}), that if $\Dc$ is a distributive lattice, then $\dim(K[\Dc]/I_{\Dc})$ is equal to the number of the join irreducible 
elements of $\Dc$ plus $1.$ Therefore, we get
\[
\dim(K[L_k]/P)=n=\dim(K[L_k]/P_1)=\dim(K[L_k]/P_1^\prime),
\]
\[
\dim(K[L_k]/P_2)=n-k, \dim(K[L_k]/P_2^\prime)=k,
\]
\[
 \dim(K[L_k]/P_3)=n-k+1,  \dim(K[L_k]/P_3^\prime)=k+1.
\] The above equalities yield the desired statements.
\end{proof}

\end{document}